\documentclass[a4paper,10pt]{article}
\usepackage[utf8]{inputenc}

\usepackage{xcolor}

\usepackage{blkarray}
\usepackage{amsmath}
\usepackage{amssymb}
\usepackage{amsthm}
\usepackage{accents}

\usepackage{graphicx}
\usepackage{caption}
\usepackage{subcaption}

\usepackage{natbib}

\numberwithin{equation}{section}

\newtheorem{theorem}{Theorem}

\newtheorem{remark}{Remark}
\newtheorem{proposition}{Proposition}

\title{A Linear Algebraic Truncation Algorithm with\\ A Posteriori Error Bounds for Computing Markov Chain Equilibrium Gradients}
\author{Saied Mahdian and Peter W. Glynn\\ Department of Management Science and Engineering,\\ Stanford University, CA, 94305, USA}
\date{}

\begin{document}

\maketitle

\begin{abstract}
The numerical computation of equilibrium reward gradients for Markov chains appears in many applications for example within the policy improvement step arising in connection with average reward stochastic dynamic programming. When the state space is large or infinite, one will typically need to truncate the state space in order to arrive at a numerically tractable formulation. In this paper, we derive the first computable a posteriori error bounds for equilibrium reward gradients that account for the error induced by the truncation. Our approach uses regeneration to express equilibrium quantities in terms of the expectations of cumulative rewards over regenerative cycles. Lyapunov functions are then used to bound the contributions to these cumulative rewards and their gradients from path excursions that take the chain outside the truncation set. Our numerical results indicate that our approach can provide highly accurate bounds with truncation sets of moderate size. We further extend our approach to Markov jump processes.

\end{abstract}

\section{Introduction}

Let $X = (X_n: n \geq 0)$ be a Markov chain with a large finite or infinite discrete state space $S$. Suppose that the dynamics of $X$ are affected by a scalar decision or control parameter $\theta$ taking values in some open interval $\Theta$. (Of course, the theory then trivially extends to vector decision parameters, since the computation of each partial derivative is a scalar problem.) Then, the one-step transition matrix $P(\theta) = (P(\theta,x,y): x,y \in S)$ depends on $\theta$. Given a reward function $r: S \to \mathbb{R}_+$, it is often of importance to analyze the equilibrium (or steady state) reward per unit time, given by 
\begin{align*}
   \alpha(\theta) = \sum_{x \in S} \pi(\theta,x) r(x),
\end{align*}
where $\pi(\theta) = (\pi(\theta,x): x \in S)$ is the equilibrium distribution of $X$. Such an equilibrium distribution exists uniquely when $X$ is irreducible and positive recurrent under $P(\theta)$. When $r$ is of mixed sign, we can apply our methodology to the positive and negative parts of $r$.

In many settings, it is also important to compute the derivative $\alpha'(\theta)$ with respect to $\theta$, assuming that it exists. Such derivatives appear naturally in the numerical optimization of Markov chain models and in the statistical analysis of Markov chains and their associated sensitivity analysis; see \cite{glynn1990likelihood} for a discussion. This gradient information also plays a key role within the policy improvement step found within algorithms designed to compute optimal policies for average reward Markov decision processes; see, for example, \cite{sutton2018reinforcement} for applications within the context of reinforcement learning. 

When stochastic simulation is used to compute performance measures, there is already a large literature on numerical computation of such derivatives; see, for example, \cite{glynn1987likelihood,pflug1992gradient,glasserman1992derivative,glynn1995likelihood}.  
On the other hand, when linear algebraic methods are used, \cite{golub1986using}  
describes the basic computation of the derivative. However, we are unaware of any literature that discusses the truncation issue that inevitably appears when $S$ is large or infinite. This paper introduces this problem and provides the first convergent truncation algorithm with computable rigorous a posteriori error bounds. 

The approach we follow in this paper extends the ideas used in \cite{zheng2024regenerationbased} and \cite{infanger2022new}   
to derive computable error bounds for the quantity $\alpha(\theta_0)$, when based on a suitable truncation approximation. In both papers, equilibrium quantities are expressed in terms of expectations involving cycles corresponding to some fixed finite subset of $S$. \cite{zheng2024regenerationbased} uses a singleton $\{z\}$ as the fixed finite subset, whereas \cite{infanger2022new} uses a more general finite subset within their algorithm and bounds. In this paper, we also use cycles defined in terms of a singleton $\{z\}$ as the starting point for our analysis, since it avoids the complications associated with needing to compute error bounds for the gradient of the equilibrium probabilities for the chain observed at return times to $K$, as would appear if we used the more general cycle structure associated with returns to $K$. In particular, the approach we adopt here involves selecting a return state $z \in S$, and letting $\tau(z) = \inf\{n\geq1: X_n = z \}$ be the first return time to $z$. If $\mathbb{E}^{\theta}(\cdot)$ is the expectation on the path-space of $X$ under which $X$ evolves according to the transition matrix $P(\theta)$ and $\mathbb{E}^{\theta}_x(\cdot) \overset{\Delta}{=} \mathbb{E}^{\theta}(\cdot| X_0 =x )$ for $x\in S$, it is well known from the theory of regeneration that $\alpha(\theta)$ can be expressed as  
\begin{align}
\alpha(\theta) = \frac{\mathbb{E}^{\theta}_z \sum_{j=0}^{\tau(z)-1} r(X_j)}{\mathbb{E}^{\theta}_z \tau(z)};
\label{eq:eqn1_1}
\end{align}
see, for example, \citet{asmussen2003applied}.
We then show how the numerator and denominator and their derivatives can be well approximated via a suitable truncation of the state space $S$. Furthermore, we are able to obtain a posteriori error bounds on the numerator and denominator of (\ref{eq:eqn1_1}), as well as their derivatives, thereby yielding computable error bounds on the derivative $\alpha'(\theta_0)$ for $\theta_0 \in \Theta$.  

The use of (\ref{eq:eqn1_1}), rather than working directly with the stationary equation $\pi(\theta) = \pi(\theta) P(\theta)$ (where $\pi(\theta)$ is encoded as a row vector), provides an expression for $\alpha(\theta)$ involving path \emph{excursions} from $z$ back to $z$ that allow one to use Lyapunov functions to numerically bound the path expectations that arise in computing the derivative.
 Thus, (\ref{eq:eqn1_1}) plays a key role in our approach. The resulting approximation to $\alpha'(\theta_0)$ is, in significant generality, convergent to $\alpha'(\theta_0)$ as the truncation set expands to the entire state space $S$; see Proposition \ref{prop1}.

\section{The Truncation Approximation to the Gradient}\label{sec:trunc_approx_grads}

Let $e: S \to \mathbb{R}_+$ be the function for which $e(x) = 1$ for $x \in S$ (the \emph{all one's} function). Also, for $f: S \to \mathbb{R}$ and $x \in S$, let 
\begin{align}
 w(\theta,x,f) = \mathbb{E}_x^{\theta} \sum_{j=0}^{\tau(z)-1} f(X_j). 
\label{eq:eqn2_1}
\end{align}
We note that the approach outlined in the Introduction for computing our truncation approximation to $\alpha'(\theta)$ involves computing approximations to $w(\theta,z,f)$ and $w'(\theta_0,z,f)$ for $f=r$ and $f=e$. Consequently, we show now how to approximate $w(\theta,z,f)$ and $w'(\theta_0,z,f)$ for generic $f$. We assume throughout this paper that $X$ is an irreducible recurrent Markov chain for $\theta \in \Theta$.

Let $A \subseteq S$ be a finite set for which $z$ is an element. Put $\kappa = A - \{z\}$. We write $P(\theta)$ and $f$ in partitioned form as  
\begin{align}
P(\theta) = \begin{blockarray}{cccc}
                 & \{z\} & \kappa & A^c \\
                 \begin{block}{c[ccc]}
                 \{z\} & P_{11}(\theta) & P_{1\kappa}(\theta) & P_{14}(\theta) \\
                 \kappa & P_{\kappa1}(\theta) & P_{\kappa\kappa}(\theta) & P_{\kappa4}(\theta) \\
                 A^c & P_{41}(\theta) & P_{4\kappa}(\theta) & P_{44}(\theta) \\
                 \end{block}
              \end{blockarray},
f = \begin{blockarray}{cc}
                     \\
                 \begin{block}{c[c]}
                 \{z\} & f_{1}  \\
                 \kappa & f_{\kappa}  \\
                 A^c & f_{4}  \\
                 \end{block}
              \end{blockarray}.
\end{align}
(We index $A^c$ as $4$ to be consistent with Section \ref{sec:computable_error_bounds}.)
If $T = \inf\{ n \geq 0: X_n \in A^c \}$, we can approximate $w(\theta,x,f)$ for $x \in A$ via  
\begin{align}
 \undertilde{w}(\theta,x,f) = \mathbb{E}_x^{\theta} \sum_{j=0}^{(\tau(z) \wedge T) -1} f(X_j).
\label{eq:eqn__2_2}
\end{align}
But
\begin{align}
 \undertilde{w}(\theta,x,f) = f(x) + \sum_{y \in \kappa}P(\theta,x,y)\undertilde{w}(\theta,y,f).
\label{eq:eqn2_4}
\end{align}
for $x \in \kappa$.
So, if $\undertilde{w}_{\kappa}(\theta,f) = (\undertilde{w}(\theta,x,f): x \in \kappa)$ is encoded as a column vector, (\ref{eq:eqn2_4}) implies that 
\begin{align*}
 \undertilde{w}_{\kappa}(\theta,f) = f_{\kappa} + P_{\kappa\kappa}(\theta) \undertilde{w}_{\kappa}(\theta,f)
\end{align*}
so that
\begin{align*}
 \undertilde{w}_{\kappa}(\theta,f) = \sum_{n=0}^{\infty} P^n_{\kappa\kappa}(\theta) f_{\kappa}.
\end{align*}
When $P(\theta)$ is irreducible, it follows that $P^n_{\kappa\kappa}(\theta) \to 0$ as $n \to \infty$ so that
\begin{align*}
 \undertilde{w}_{\kappa}(\theta,f) = (I - P_{\kappa\kappa}(\theta))^{-1} f_{\kappa};
\end{align*}
see, for example, \citet{kemeny1960finite}. So, we arrive at the approximation
\begin{align}
 \undertilde{w}(\theta,z,f) = f(z) + P_{1\kappa}(\theta) \mathcal{I}_{\kappa\kappa}(\theta) f_{\kappa},
\label{eq:eqn2_5}
\end{align}
to $w(\theta,z,f)$, where $\mathcal{I}_{\kappa\kappa}(\theta)= (I - P_{\kappa\kappa}(\theta) )^{-1} = \sum_{n=0}^{\infty} P^n_{\kappa\kappa}(\theta)$. 
The approximation to $w'(\theta_0,z,f)$ is therefore given by
\begin{align}
 \undertilde{w}'(\theta_0,z,f) =& P'_{1\kappa}(\theta_0) \mathcal{I}_{\kappa\kappa}(\theta_0) f_{\kappa} \nonumber \\
                                       & + P_{1\kappa}(\theta_0) \mathcal{I}_{\kappa\kappa}(\theta_0) P'_{\kappa\kappa}(\theta_0) \mathcal{I}_{\kappa\kappa}(\theta_0) f_{\kappa}.
\label{eq:eqn2_6}
\end{align}

We now prove that as the truncation set $A$ expands to the entire state space $S$, our approximations converge. In particular, suppose that $z \in A_1$ and that $A_1 \subseteq A_2 \subseteq A_3 \subseteq \ldots$ such that $\bigcup_{n=1}^{\infty} A_n = S$. Let $\undertilde{w}_n(\theta_0,z,f)$ and $\undertilde{w}'_n(\theta_0,z,f)$ be the truncation approximations corresponding to setting $A = A_n$. Under suitable regularity conditions, it is known that 
\begin{align*}
 w(\theta,z,f) = \mathbb{E}_z^{\theta_0} \sum_{j=0}^{\tau(z)-1}f(X_j) \prod_{k=0}^{j-1} \frac{P(\theta,X_k,X_{k+1})}{P(\theta_0,X_k,X_{k+1})}, 
\end{align*}
and
\begin{align}
 w'(\theta_0,z,f) &= \mathbb{E}_z^{\theta_0} \sum_{j=0}^{\tau(z)-1}f(X_j) \sum_{k=0}^{j-1} \frac{P'(\theta_0,X_k,X_{k+1})}{P(\theta_0,X_k,X_{k+1})} \nonumber \\
 &= \mathbb{E}_z^{\theta_0} \sum_{j=0}^{\tau(z)-1}  \frac{P'(\theta_0,X_j,X_{j+1})}{P(\theta_0,X_j,X_{j+1})} \sum_{k=j+1}^{\tau(z)-1} f(X_k), 
\label{eq:eqn2_7}
\end{align}
where
\begin{align}
 \mathbb{E}_z^{\theta_0} \sum_{j=0}^{\tau(z)-1} |f(X_j)| \sum_{k=0}^{j-1} \frac{|P'(\theta_0,X_k,X_{k+1})|}{P(\theta_0,X_k,X_{k+1})} < \infty; 
\label{eq:eqn2_8}
\end{align}
see \citet{glynn1995likelihood} and \citet{rhee2023lyapunov}. Put $|f| = (|f|(x): x \in S)$.

\begin{proposition}
\begin{enumerate}
  \item[(a)] If $w(\theta_0, z, |f|) < \infty$, then  $\undertilde{w}_n(\theta_0, z, f) \to w(\theta_0, z, f)$ as $n \to \infty$.
  \item[(b)] If (\ref{eq:eqn2_8}) holds, then $\undertilde{w}'_n(\theta_0, z, f) \to w'(\theta_0, z, f)$ as $n \to \infty$.
\end{enumerate}
\label{prop1}
\end{proposition}

\begin{proof}
 Let $T_n = \inf\{ k\geq0: X_k \in A^c_n  \}$, and note that $T_n \to \infty$. Also, (\ref{eq:eqn__2_2}) implies that  
\begin{align*}
 \undertilde{w}_n(\theta,z,f) = \mathbb{E}_z^{\theta} \sum_{j=0}^{(\tau(z) \wedge T_n) -1} f(X_j).
\end{align*}
But $\sum_{j=0}^{(\tau(z) \wedge T_n) -1} f(X_j) \to \sum_{j=0}^{\tau(z) -1} f(X_j)$ a.s. as $n \to \infty$, and 
\begin{align*}
\left| \sum_{j=0}^{(\tau(z) \wedge T_n) -1} f(X_j)\right| \leq \sum_{j=0}^{\tau(z) -1} |f(X_j)|,
\end{align*}
so that the Dominated Convergence Theorem yields (a) in the presence of $w(\theta_0, z, |f|) < \infty$.  

As for (b), it is easily verified that  
\begin{align*}
 \undertilde{w}'(\theta_0,z,f) = \mathbb{E}_z^{\theta_0} \sum_{j=0}^{(\tau(z)\wedge T)-1}f(X_j) \sum_{k=0}^{j-1} \frac{P'(\theta_0,X_k,X_{k+1})}{P(\theta_0,X_k,X_{k+1})}, 
\end{align*}
so that
\begin{align*}
 \undertilde{w}'_n(\theta_0,z,f) = \mathbb{E}_z^{\theta_0} \sum_{j=0}^{(\tau(z)\wedge T_n)-1}f(X_j) \sum_{k=0}^{j-1} \frac{P'(\theta_0,X_k,X_{k+1})}{P(\theta_0,X_k,X_{k+1})}. 
\end{align*}
Again,
\begin{align*}
  \sum_{j=0}^{(\tau(z)\wedge T_n)-1}f(X_j) \sum_{k=0}^{j-1} \frac{P'(\theta_0,X_k,X_{k+1})}{P(\theta_0,X_k,X_{k+1})} \to \sum_{j=0}^{\tau(z)-1}f(X_j) \sum_{k=0}^{j-1} \frac{P'(\theta_0,X_k,X_{k+1})}{P(\theta_0,X_k,X_{k+1})}
\end{align*}
as $n \to \infty$. In view of (\ref{eq:eqn2_8}), the Dominated Convergence Theorem yields (b). 
\end{proof}

\section{A General Lyapunov Bound for Solutions to Non-negative Linear Systems}\label{sec:general_bound_lyapunov}

To develop a posteriori error bounds on $\undertilde{w}(\theta_0,z,f)$ and $\undertilde{w}'(\theta_0,z,f)$, it is useful to provide some general theory in this section on bounds for solutions to linear systems in which the coefficient matrix has a specific non-negative structure. Put $\Lambda = S - \{z\}$, $A' = A - \{z\}$, and let $G = (G(x,y): x,y \in \Lambda)$ be a non-negative matrix and $h = (h(x): x \in  \Lambda)$ be a (column) vector. Suppose $K' \subseteq A' \subseteq \Lambda$ with $|A'| < \infty$, set $\tilde{A} = A' - K'$, and write $G$ and $h$ in partitioned form as
\begin{align*}
G = \begin{blockarray}{cccc}
                 & K' & \tilde{A} & A^c \\
                 \begin{block}{c[ccc]}
                 K' & G_{22} & G_{23} & G_{24} \\
                 \tilde{A} & G_{32} & G_{33} & G_{34} \\
                 A^c & G_{42} & G_{43} & G_{44} \\
                 \end{block}
              \end{blockarray},
h = \begin{blockarray}{cc}
                     \\
                 \begin{block}{c[c]}
                 K' & h_2  \\
                 \tilde{A} & h_3  \\
                 A^c & h_4  \\
                 \end{block}
              \end{blockarray}.
\end{align*}
(We index $K',\tilde{A}$, and $A^c$ as $2,3$ and $4$ in order to be consistent with the discussion of Section \ref{sec:computable_error_bounds}.)
The theory in this section addresses the solution $u^T=[u^T_2,u^T_3,u^T_4]$ to
\begin{align*}
  u = h + G u,
\end{align*}
as given by
\begin{align}
  u^* = \sum_{n=0}^{\infty} G^n h. 
\label{eq:eqn3__1}
\end{align}

Let $\kappa = K' \cup \tilde{A}$ and consider the partitioned submatrix and vector given by
\begin{align*}
G_{\kappa\kappa} = \begin{bmatrix} 
                 G_{22} & G_{23} \\
                 G_{32} & G_{33} 
              \end{bmatrix},
h_{\kappa} = \begin{bmatrix} 
                 h_{2} \\
                 h_{3} 
              \end{bmatrix}.  
\end{align*}
If $G^n_{\kappa\kappa} \to 0$ as $n \to \infty$, $\mathcal{I}_{\kappa\kappa} = (I - G_{\kappa\kappa})^{-1}$ exists; see \citet{kemeny1960finite}. Write $\mathcal{I}_{\kappa\kappa}$ in row-partitioned form as
\begin{align*}
\mathcal{I}_{\kappa\kappa} =   
 \begin{blockarray}{cc}
                     \\
                 \begin{block}{c[c]}
                 K' & \mathcal{I}_{2}  \\
                 \tilde{A} & \mathcal{I}_3  \\
                 \end{block}
              \end{blockarray}.
\end{align*}
Then, a suitable finite-dimensional approximation $\undertilde{u}^T = [\undertilde{u}^T_2, \undertilde{u}^T_3,\undertilde{u}^T_4]$ to $u^*$ is given by
\begin{align*}
\undertilde{u} =   
 \begin{blockarray}{cc}
                     \\
                 \begin{block}{c[c]}
                 K' & \mathcal{I}_{2} h_{\kappa}  \\
                 \tilde{A} & \mathcal{I}_3 h_{\kappa} \\
                 A^c & 0  \\
                 \end{block}
              \end{blockarray}.
\end{align*}
Our goal, in the remainder of this section, is to develop a suitable error bound on $\undertilde{u}$ as an approximation to $u^*$.

To accomplish this, we let $\beta = \tilde{A} \cup A^c = \Lambda - K'$ and set
\begin{align*}
G_{\beta\beta} = \begin{bmatrix} 
                 G_{33} & G_{34} \\
                 G_{43} & G_{44} 
              \end{bmatrix},  
h_{\beta} = \begin{bmatrix} 
                 h_{3} \\
                 h_{4} 
              \end{bmatrix}.  
\end{align*}
Let $w = (w(x): x \in K')$ be a positive function on $K'$. The function $w$ induces a norm on functions/column vectors $g =( g(x): x \in K')$ with domain $K'$, via $\| g \|_w = \max\{|g(x)|/w(x): x \in K' \}$. We now assume existence of non-negative functions $v$ and $\nu$ (known as \emph{Lyapunov functions}) such that:
\begin{align}
G_{\beta\beta} v &\leq v - |h_{\beta}|, \nonumber\\
G_{\beta\beta} \nu &\leq \nu - G_{\beta2}w.
\label{eq:eqn3__2}
\end{align}
In the presence of (\ref{eq:eqn3__2}), it is well known that
\begin{align}
\sum_{n=0}^{\infty} G^n_{\beta\beta} |h_{\beta}| \leq v, \nonumber\\
\sum_{n=0}^{\infty} G^n_{\beta\beta} G_{\beta2} w \leq \nu; 
\label{eq:eqn3__3}
\end{align}
see, for example, \citet{infanger2024posteriori}.

\begin{remark}
If $G$ is substochastic, then $\sum_{n=1}^{\infty} G^n_{\beta\beta} G_{\beta2}$ is substochastic. Then, if we let $e_2 = (e_2(x): x \in K')$ and $e_{\beta} = (e_{\beta}: x \in \beta)$ be the \emph{all ones} vectors on $K'$ and $\beta$ respectively, we can set $w = e_2$ and $\nu = e_{\beta}$. With these choices, only the Lyapunov function $v$ needs to be constructed.
\label{rmk_1}
\end{remark}

Put $\epsilon = u^* - \undertilde{u}$ and let $|\epsilon| = (|\epsilon(x)|: x \in \Lambda )$. Write $|\epsilon|$ in partitioned form as $|\epsilon|^T = [|\epsilon_2|^T,|\epsilon_3|^T,|\epsilon_4|^T]$.

\begin{theorem}
Suppose $G^n_{\kappa\kappa} \to 0$ as $n \to \infty$, and assume the existence of $v^T = [v^T_3, v^T_4]$ and $\nu^T = [\nu^T_3, \nu^T_4]$ satisfying (\ref{eq:eqn3__2}). If $\| \mathcal{I}_2 G_{24} \nu_4 \|_w < 1$, then 
\begin{align*}
 |\epsilon_i| \leq \mathcal{I}_i G_{i4} v_4 + \mathcal{I}_i G_{i4} \nu_4 \cdot m(h).
\end{align*}
for $i = 2,3$ and
\begin{align*}
 |\epsilon_4| \leq v_4 + \nu_4 \cdot m(h),
\end{align*}
where 
\begin{align*}
  m(h) = \frac{\| \mathcal{I}_2 G_{24}v_4 + \mathcal{I}_2 |h_{2}|  \|_w}{1 - \|\mathcal{I}_2 G_{24} \nu_4  \|_w }.
\end{align*}

\label{thm1}
\end{theorem}

\begin{proof}
Put 
\begin{align*}
u^*_{\beta} = \begin{bmatrix} 
                 u^*_{3}\\
                 u^*_{4}
              \end{bmatrix},
G_{\beta2} = 
\begin{bmatrix}
G_{32}\\
G_{42}
\end{bmatrix},
\end{align*}
and note that
\begin{align*}
  u^*_{\beta} = G_{\beta2} u^*_2 + G_{\beta\beta} u^*_{\beta} + h_{\beta}.
\end{align*}
It follows that
\begin{align*}
  u^*_{\beta} = \sum_{n=0}^{\infty} G^n_{\beta\beta}  (G_{\beta2} u^*_2 + h_{\beta}).
\end{align*}

Since $|u^*_2(x)| \leq w(x) \|u^*_2 \|_w$, (\ref{eq:eqn3__3}) and the non-negativity of $G$ imply that
\begin{align*}
 |u^*_{\beta}| \leq \nu \|u^*_2 \|_w + v,
\end{align*}
where $|u^*_{\beta}| = (|u^*(x)|: x \in \beta)$. So,
\begin{align}
  |u^*_4 | \leq  \nu_4 \|u^*_2 \|_w + v_4.
\label{eq:eqn3__4}
\end{align}
Put
\begin{align*}
u^*_{\kappa} = \begin{bmatrix} 
                 u^*_2\\
                 u^*_3
              \end{bmatrix},
G_{\kappa4} = \begin{bmatrix} 
                 G_{24}\\
                 G_{34}
              \end{bmatrix},
\undertilde{u}_{\kappa} = \begin{bmatrix} 
                 \undertilde{u}_2\\
                 \undertilde{u}_3
              \end{bmatrix},
{\epsilon}_{\kappa} = \begin{bmatrix} 
                 {\epsilon}_2\\
                 {\epsilon}_3
              \end{bmatrix}.
\end{align*}
Since
\begin{align}
  u^*_{\kappa} = G_{\kappa\kappa} u^*_{\kappa} + G_{\kappa4} u^*_{4} + h_{\kappa},
\label{eq:eqn3__5}
\end{align}
it follows that 
\begin{align*}
  u^*_2 = \mathcal{I}_2 (G_{24}u^*_4 + h_{2}).
\end{align*}

In view of the fact that $(I - G_{\kappa\kappa})^{-1} = \sum_{n=0}^{\infty} G^n_{\kappa\kappa}$ with $G_{\kappa\kappa}$ non-negative, we find that
\begin{align*}
 | u^*_2 | \leq \mathcal{I}_2 G_{24} (\nu_4 \| u^*_2 \|_w + v_4) + \mathcal{I}_2 |h_{2}|.
\end{align*}
Consequently,
\begin{align*}
  \| u^*_2 \|_w \leq \| \mathcal{I}_2 G_{24} \nu_4 \|_w \cdot \| u^*_2 \|_w + \|\mathcal{I}_2 G_{24}v_4 + \mathcal{I}_2 |h_{2}| \|_w;
\end{align*}
and hence
\begin{align}
 \| u^*_2  \| \leq m(h).
\label{eq:eqn3__6}
\end{align}

Also,
\begin{align}
 \undertilde{u}_{\kappa} = G_{\kappa\kappa} \undertilde{u}_{\kappa} + h_{\kappa}.
\label{eq:eqn3__7}
\end{align}
The relations (\ref{eq:eqn3__5}) and (\ref{eq:eqn3__7}) imply that
\begin{align}
 \epsilon_{\kappa} = G_{\kappa\kappa} \epsilon_{\kappa} + G_{\kappa4} u^*_4
\label{eq:eqn3__8}
\end{align}
So, (\ref{eq:eqn3__4}),(\ref{eq:eqn3__6}),(\ref{eq:eqn3__7}) and (\ref{eq:eqn3__8}) show that
\begin{align*}
 | \epsilon_i | \leq \mathcal{I}_i  G_{i4} (\nu_4 \cdot m(h) + v_4)
\end{align*}
for $i = 2,3$, and
\begin{align*}
|\epsilon_4| \leq \nu_4 \cdot m(h) + v_4,
\end{align*}
proving the theorem.
\end{proof}

\section{Computable Error Bounds for the Truncation Approximations}\label{sec:computable_error_bounds}

In this section, we apply the general bounds developed in Section \ref{sec:general_bound_lyapunov} to the task of developing computable bounds on $|w(\theta_0,z,f)-\undertilde{w}(\theta_0,z,f)|$ and $|w'(\theta_0,z,f)-\undertilde{w}'(\theta_0,z,f)|$. As noted in Section 2, we will be applying this theory to $f=r$ and $f=e$, so that we may assume in this section that $f$ is non-negative. (Recall from Section 1 that we are assuming throughout this paper that $r$ is non-negative.)
We start by writing
\begin{align*}
P(\theta_0) = \begin{blockarray}{ccc}
                 & \{z\} & \Lambda \\
                 \begin{block}{c[cc]}
                   \{z\} & P_{11}(\theta_0) & P_{1\Lambda}(\theta_0)  \\
                 \Lambda & P_{\Lambda1}(\theta_0) & P_{\Lambda\Lambda}(\theta_0)  \\
                 \end{block}
              \end{blockarray},
f = \begin{blockarray}{cc}
                     \\
                 \begin{block}{c[c]}
                 \{z\} & f_1 \\
                 \Lambda & f_{\Lambda}  \\
                 \end{block}
              \end{blockarray}.
\end{align*}
where
\begin{align*}
P_{\Lambda\Lambda}(\theta_0) = \begin{blockarray}{cccc}
                 & K' & \tilde{A} & A^c \\
                 \begin{block}{c[ccc]}
                 K' &  P_{22}(\theta_0) & P_{23}(\theta_0) & P_{24}(\theta_0) \\
                 \tilde{A} &  P_{32}(\theta_0) & P_{33}(\theta_0) & P_{34}(\theta_0) \\
                 A^c  & P_{42}(\theta_0) & P_{43}(\theta_0) & P_{44}(\theta_0) \\
                 \end{block}
              \end{blockarray},
f_{\Lambda} = \begin{blockarray}{cc}
                     \\
                 \begin{block}{c[c]}
                 K' & f_2  \\
                 \tilde{A} & f_3  \\
                 A^c & f_{4}  \\
                 \end{block}
              \end{blockarray}.
\end{align*}
We note that (\ref{eq:eqn2_1}) implies that
\begin{align}
 w(\theta_0, z,f) = f(z) + P_{1\Lambda}(\theta_0) \sum_{n=0}^{\infty} P^n_{\Lambda\Lambda}(\theta_0) f_{\Lambda}
\end{align}
and, (\ref{eq:eqn2_7}) and (\ref{eq:eqn2_8}) imply that 
\begin{align}
 w'(\theta_0, z, f) =& P'_{1\Lambda}(\theta_0) \mathcal{I}_{\Lambda\Lambda}(\theta_0) f_{\Lambda} \nonumber \\
                    &+ P_{1\Lambda}(\theta_0) \mathcal{I}_{\Lambda\Lambda}(\theta_0) P'_{\Lambda\Lambda}(\theta_0)  \mathcal{I}_{\Lambda\Lambda}(\theta_0) f_{\Lambda},
\label{eq:eqn4__2}
\end{align}
where $\mathcal{I}_{\Lambda\Lambda}(\theta_0) = \sum_{n=0}^{\infty} P^n_{\Lambda\Lambda}(\theta_0)$.

Also, our approximations can be rewritten as 
\begin{align*}
 \undertilde{w}(\theta_0, z, f) = f(z) + P_{1\kappa}(\theta_0) \mathcal{I}_{\kappa\kappa}(\theta_0) f_{\kappa}
\end{align*}
and
\begin{align*}
 \undertilde{w}'(\theta_0, z, f) =& P'_{1\kappa}(\theta_0) \mathcal{I}_{\kappa\kappa}(\theta_0) f_{\kappa}\\
                    &+ P_{1\kappa}(\theta_0) \mathcal{I}_{\kappa\kappa}(\theta_0) P'_{\kappa\kappa}(\theta_0)  \mathcal{I}_{\kappa\kappa}(\theta_0) f_{\kappa}
\end{align*}
where $\mathcal{I}_{\kappa\kappa}(\theta_0) = \sum_{n=0}^{\infty} P^n_{\kappa\kappa}(\theta_0) = (I - P_{\kappa\kappa}(\theta_0))^{-1}$; see (\ref{eq:eqn2_5}) and (\ref{eq:eqn2_6}). Write
\begin{align*}
\mathcal{I}_{\kappa\kappa}(\theta_0) = \begin{blockarray}{cc}
                     \\
                 \begin{block}{c[c]}
                 K' & \mathcal{I}_2(\theta_0)  \\
                 \tilde{A} & \mathcal{I}_3(\theta_0)  \\
                 \end{block}
              \end{blockarray}.
\end{align*}
Put
\begin{align*}
\undertilde{w}(\theta_0,f) = \begin{blockarray}{cc}
                     \\
                 \begin{block}{c[c]}
                 \{z\} & \undertilde{w}(\theta_0,z,f)  \\
                 \kappa & \undertilde{w}_{\kappa}(\theta_0,f)  \\
                 A^c & 0  \\
                 \end{block}
              \end{blockarray} = 
\begin{blockarray}{cc}
       \\
   \begin{block}{c[c]}
   \{z\} & \undertilde{w}(\theta_0,z,f)  \\
   \Lambda & \undertilde{w}_{\Lambda}(\theta_0,f)  \\
   \end{block}
\end{blockarray}, 
\end{align*}
\begin{align*}
   w_{\Lambda}(\theta_0,f) = \mathcal{I}_{\Lambda\Lambda}(\theta_0) f_{\Lambda},
\end{align*}
\begin{align}
   w'_{\Lambda}(\theta_0,f) = \mathcal{I}_{\Lambda\Lambda}(\theta_0) P'_{\Lambda\Lambda}(\theta_0) \mathcal{I}_{\Lambda\Lambda}(\theta_0)f_{\Lambda},
\end{align}
\begin{align*}
  \epsilon_{\Lambda}(f)^T = w_{\Lambda}(\theta_0,f)^T - \undertilde{w}_{\Lambda}(\theta_0,f)^T = [ \epsilon_2(f)^T, \epsilon_3(f)^T, \epsilon_4(f)^T],  
\end{align*}
\begin{align*}
  \epsilon'_{\Lambda}(f)^T = w'_{\Lambda}(\theta_0,f)^T - \undertilde{w}'_{\Lambda}(\theta_0,f)^T = [ \epsilon'_2(f)^T, \epsilon'_3(f)^T, \epsilon'_4(f)^T].  
\end{align*}

Let
\begin{align*}
P_{\beta\beta}(\theta_0) = \begin{bmatrix}
                               P_{33}(\theta_0) & P_{34}(\theta_0)  \\
                               P_{43}(\theta_0) & P_{44}(\theta_0)  \\
                            \end{bmatrix},
\end{align*}
and write $e^T = [e^T_1, e^T_2, e^T_3, e^T_4]$ in partitioned form. 
Because we are assuming that $P(\theta_0)$ is irreducible, $P^n_{\kappa\kappa}(\theta_0) \to 0$ as $n \to \infty$.
Then, Theorem \ref{thm1} immediately implies the following result when applied to $\epsilon_{\Lambda}(f)$.

\begin{proposition}
Assume that $f$ is non-negative and suppose that $\|\mathcal{I}_{2}(\theta_0) P_{\kappa4}(\theta_0) e_4 \|_{e_2} < 1$ and that there exists a non-negative function $v = (v(x): x \in \beta)$ such that
\begin{align}
 P_{\beta\beta}(\theta_0) v \leq v - f_{\beta}.
\label{eqn:eqn4_4}
\end{align}
Then,
\begin{align*}
  | \epsilon_i(f) | \leq \mathcal{I}_i(\theta_0) P_{i4}(\theta_0) v_4 + \mathcal{I}_i(\theta_0) P_{i4}(\theta_0) e_4 \cdot m(f) \overset{\Delta}{=} \mu_i(f)
\end{align*}
for $i = 2,3$,
\begin{align*}
 | \epsilon_4(f) | \leq v_4 + e_4 \cdot m(f) \overset{\Delta}{=} \mu_4(f)
\end{align*}
and
\begin{align*}
 m(f) = \frac{\|\mathcal{I}_{2}(\theta_0)P_{24}(\theta_0) v_4 + \mathcal{I}_{2}(\theta_0)|f_{2} |   \|_{e_2}}{1 - \|\mathcal{I}_{2}(\theta_0) P_{24}(\theta_0) e_4  \|_{e_2}}.
\end{align*}
\label{prop2}

\end{proposition}

Proposition \ref{prop2} clearly implies that
\begin{align*}
 |w(\theta_0,z,f)-\undertilde{w}(\theta_0,z,f)| \leq \sum_{i=2}^4 P_{1i}(\theta_0)\mu_i(f) \overset{\Delta}{=} \mu_1(f).
\end{align*}

For $1 \leq i \leq 4$, put
\begin{align*}
  \tilde{w}_i(\theta_0,f) = \undertilde{w}_i(\theta_0,f) + \mu_i(f)
\end{align*}
and note that
\begin{align*}
  \undertilde{w}_i(\theta_0,f) \leq w_i(\theta_0,f) \leq \tilde{w}_i(\theta_0,f).
\end{align*}

Put $T_K = \inf\{n>T: X_n \in K \}$. To construct appropriate error bounds on the derivative $w'(\theta_0,f)$, observe that for $x\in S$,
\begin{align}
 w'(\theta_0,x,f) =& \undertilde{w}'(\theta_0,x,f) \nonumber \\
                   &+ \mathbb{E}_x^{\theta_0} \sum_{j=0}^{T-1} \frac{P'(\theta_0,X_j,X_{j+1})}{P(\theta_0,X_j,X_{j+1})} \sum_{l=T}^{T_K-1} f(X_l) I(T < \tau(z)) \nonumber\\
                   &+ \mathbb{E}_x^{\theta_0} \sum_{j=0}^{T-1} \frac{P'(\theta_0,X_j,X_{j+1})}{P(\theta_0,X_j,X_{j+1})} \sum_{l=T_K}^{\tau(z)-1} f(X_l) I(T < \tau(z), X_{T_K} \neq z) \nonumber\\
                   &+ \mathbb{E}_x^{\theta_0} \sum_{j=T}^{T_K-1} \frac{P'(\theta_0,X_j,X_{j+1})}{P(\theta_0,X_j,X_{j+1})} \sum_{l=j+1}^{T_K-1} f(X_l) I(T < \tau(z) ) \nonumber\\
                   &+ \mathbb{E}_x^{\theta_0} \sum_{j=T}^{T_K-1} \frac{P'(\theta_0,X_j,X_{j+1})}{P(\theta_0,X_j,X_{j+1})} \sum_{l=T_K}^{\tau(z)-1} f(X_l) I(T < \tau(z), X_{T_K} \neq z) \nonumber\\
                   &+ \mathbb{E}_x^{\theta_0} w'(\theta_0,X_{T_K},f) I(T<\tau(z), X_{T_K} \neq z) \label{eqn:eqn4_5} \\ 
\overset{\Delta}{=}& \undertilde{w}'(\theta_0,x,f) + (1) + (2) + (3) + (4) \nonumber \\
                   & + \mathbb{E}_x^{\theta_0} w'_2(\theta_0,X_{T_K},f) I(T<\tau(z), X_{T_K} \neq z). \nonumber
\end{align}
It follows that for $x \in \Lambda$,
\begin{align*}
  \epsilon'(x,f) =& (1) + (2) + (3) + (4) \\
                 &+  \mathbb{E}_x^{\theta_0} \epsilon'(X_{T_K},f) I(T<\tau(z), X_{T_{K}} \neq z)\\
                 &+  \mathbb{E}_x^{\theta_0} \undertilde{w}'_2(\theta_0, X_{T_K},f) I(T<\tau(z), X_{T_{K}} \neq z)\\
                 =& (1) + (2) + (3) + (4) + \mathbb{E}_x^{\theta_0} \epsilon'_2(X_{T_K},f)I(T<\tau(z), X_{T_K} \neq z)\\ 
                 &+  \mathbb{E}_x^{\theta_0} \undertilde{w}'_2(\theta_0, X_{T_K},f) I(T<\tau(z), X_{T_{K}} \neq z).
\end{align*}

For $x \in \Lambda$, the absolute value of (1) can be bounded as follows:
\begin{align*}
|(1)| \leq& \mathbb{E}_x^{\theta_0} \sum_{j=0}^{(T \wedge \tau(z))-1} \left| \frac{P'(\theta_0, X_j, X_{j+1})}{P(\theta_0, X_j, X_{j+1})}  \right| v(X_T) I(T < \tau(z)) \\ 
 =& \begin{cases}
    \left( \left(I - P_{\kappa\kappa}(\theta_0) \right)^{-1} |P'_{\kappa\kappa}(\theta_0)| \left(I - P_{\kappa\kappa}(\theta_0) \right)^{-1} P_{\kappa4}(\theta_0) v_4 \right) (x), & x \in \kappa  \\
    0,   &  x \in A^c
  \end{cases} \\
  \overset{\Delta}{=}& a(x).
\end{align*}
As for (2), note that
\begin{align*}
  |(2)| \leq& \mathbb{E}_x^{\theta_0} \sum_{j=0}^{T-1} \frac{|P'(\theta_0, X_j, X_{j+1})|}{P(\theta_0, X_j, X_{j+1})} \max_{y \in K'} \tilde{w}_2(y) I(T < \tau(z)) \\ 
 &= \begin{cases}
    \left( \left(I - P_{\kappa\kappa}(\theta_0) \right)^{-1} |P'_{\kappa\kappa}(\theta_0)| \left(I - P_{\kappa\kappa}(\theta_0) \right)^{-1} P_{\kappa4}(\theta_0) e_4 \right) (x) \max_{y \in K'} \tilde{w}_2(y), & x \in \kappa  \\
    0,   & x \in A^c
  \end{cases}\\
  \overset{\Delta}{=}& \tilde{a}(x).
\end{align*}

To bound the absolute value of (3), we assume the existence of a non-negative Lyapunov function $\tilde{v}: \beta \to \mathbb{R}_+$ such that
\begin{align}
  P_{\beta\beta}(\theta_0) \tilde{v} \leq \tilde{v} - |P'_{\beta\beta}(\theta_0)|v.
\label{eqn:eqn4_6}
\end{align}
Then,
\begin{align*}
  \mathbb{E}_x^{\theta_0} \sum_{j=0}^{T_K-1} \frac{|P'(\theta_0, X_j, X_{j+1})|}{P(\theta_0, X_j, X_{j+1})} v(X_{j+1}) \leq \tilde{v}(x)
\end{align*}
for $x \in \beta$, so that
\begin{align*}
 |(3)| \leq& 
 \begin{cases} 
   \left( (I - P_{\kappa\kappa}(\theta_0))^{-1} P_{\kappa4} \tilde{v}_{4} \right) (x), & \quad x \in \kappa\\
   \tilde{v}_4(x), & \quad x \in A^c
 \end{cases} \\
  \overset{\Delta}{=}& b(x)
\end{align*}
for $x \in \Lambda$.

For the fourth term, we assume the existence of a third Lyapunov function $\tilde{\nu}: \beta \to \mathbb{R}_+$ such that
\begin{align}
  P_{\beta\beta}(\theta_0) \tilde{\nu} \leq \tilde{\nu} - |P'_{\beta\beta}(\theta_0)| e_{\beta},
  \label{eqn:eqn4_7}
\end{align}
so that
\begin{align*}
 \mathbb{E}_x^{\theta_0} \sum_{j=0}^{T_K-1} \frac{|P'(\theta_0, X_j, X_{j+1})|}{P(\theta_0, X_j, X_{j+1})} \leq \tilde{\nu}(x)
\end{align*}
for $x \in \beta$. Consequently, 
\begin{align*}
  |(4)| \leq& 
 \begin{cases} 
   \left( (I - P_{\kappa\kappa}(\theta_0))^{-1} P_{\kappa4} \tilde{\nu}_{4} \right) (x) \cdot \max_{y \in K'} \tilde{w}_2(y), & \quad x \in \kappa\\
   \tilde{\nu}_4(x) \cdot \max_{y \in K'} \tilde{w}_2(y), & \quad x \in A^c
 \end{cases}\\ 
  \overset{\Delta}{=}& \tilde{b}(x)
\end{align*}
for $x \in \Lambda$.

Of course,
\begin{align*}
 & |\mathbb{E}_x^{\theta_0} \undertilde{w}'_2 (\theta_0, X_{T_K}, f) I(T< \tau(z), X_{T_K} \neq z)|  \\
\leq & \max_{y \in K'} |\undertilde{w}'_2(\theta_0,y,f)| \cdot P_x^{\theta_0} (T < \tau(z)) \\ 
\leq &
 \begin{cases} 
    (I - P_{\kappa\kappa}(\theta_0))^{-1} P_{\kappa4}(\theta_0)\cdot \max_{y\in K'} |\undertilde{w}'_2(\theta_0, y, f)| , &  x \in \kappa\\
    \max_{y \in K'} |\undertilde{w}'_2(\theta_0,y,f)|, &  x \in A^c
 \end{cases} \\
  \overset{\Delta}{=}& c(x)
\end{align*}
for $x \in \Lambda$.

In view of (\ref{eqn:eqn4_5}), we conclude that for $x \in \Lambda$,
\begin{align*}
 | \epsilon'(x,f)| \leq& a(x) + \tilde{a}(x) + b(x) + \tilde{b}(x) + c(x) \\
                     &+ \mathbb{E}_x^{\theta_0} |\epsilon'(X_{T_K},f)| I(T < \tau(z), X_{T_K} \neq z)\\
                     &\leq{a(x) + \tilde{a}(x) + b(x) + \tilde{b}(x) + c(x)} \\
                     &{+ \| \epsilon'_2(f) \|_{e_2} \cdot \mathbb{P}_x^{\theta_0}  (T < \tau(z)).}
\end{align*}
Hence,
\begin{align*}
  \| \epsilon'_2(f)  \|_{e_2} \leq \|a_2 + \tilde{a}_2 + b_2 + \tilde{b}_2 + c_2  \|_{e_2} + \| \epsilon'_2(f) \|_{e_2} \cdot \| \mathcal{I}_2 (\theta_0) P_{24}(\theta_0) e_4 \|_{e_2}
\end{align*}
so that
\begin{align}
 \| \epsilon'_2(f)  \|_{e_2} \leq \frac{\| a_2 + \tilde{a}_2 + b_2 + \tilde{b}_2 + c_2 \|_{e_2}}{1 -  \| \mathcal{I}_2(\theta_0)P_{24}(\theta_0)e_4  \|_{e_2}}   \overset{\Delta}{=} \tilde{m}(f).
\label{eqn:eqn4_8}
\end{align}

With the bound $\tilde{m}(f)$ in place, we find that
\\
\begin{align}
\begin{aligned}
 | \epsilon'_{\kappa}(f)| &\leq a_{\kappa} + \tilde{a}_{\kappa} + b_{\kappa} + \tilde{b}_{\kappa} + c_{\kappa}\\ 
                          &+ \tilde{m}(f)\; \mathcal{I}_{\kappa\kappa}(\theta_0) P_{\kappa4}(\theta_0) e_4 e_{\kappa}\overset{\Delta}{=} \tilde{\mu}_{\kappa}(f),\\
 | \epsilon'_{4}(f)| &\leq a_{4} + \tilde{a}_{4} + b_{4} + \tilde{b}_{4} + c_{4} + \tilde{m}(f)\;  e_{4} \overset{\Delta}{=} \tilde{\mu}_{4}(f).
\end{aligned}
\label{eqn:eqn4_9}
\end{align}
We summarize our findings with the following theorem.
\begin{theorem}
  Assume that $f$ is non-negative and suppose that $\| \mathcal{I}_2(\theta_0) P_{\kappa4}(\theta_0) e_4 \|_{e_2}<1$ and that there exist non-negative $v$, $\tilde{v}$ and $\tilde{\nu}$ satisfying (\ref{eqn:eqn4_4}), (\ref{eqn:eqn4_6}), and (\ref{eqn:eqn4_7}), respectively. Then, the bounds of (\ref{eqn:eqn4_8}) and (\ref{eqn:eqn4_9}) hold.
\label{thm2}
\end{theorem}

Since 
\begin{align*}
  w(\theta,z,f) - \undertilde{w}(\theta,z,f) = \sum_{i=2}^4 P_{1i}(\theta_0) \epsilon_i(f)
\end{align*}
we get 
\begin{align*}
  & |w'(\theta,z,f) - \undertilde{w}'(\theta,z,f)| \leq \sum_{i=2}^4 P_{1i}(\theta_0) | \epsilon'_i(f)| + \sum_{i=2}^4 |P'_{1i}(\theta_0)| \epsilon_i(f) \\
  \leq & \sum_{i=2}^4 P_{1i}(\theta_0)  \tilde{\mu}_i(f) + \sum_{i=2}^4 |P'_{1i}(\theta_0)| \mu(f).
\end{align*}

\section{Computable Error Bounds for the Derivative}\label{sec:errorbound_deriv}

In this section, we use Proposition \ref{prop2} and Theorem \ref{thm2} to develop bounds on $\alpha'(\theta_0)$. Note that
\begin{align}
  \alpha'(\theta_0) = \frac{w'(\theta_0,z,r) - \alpha(\theta_0)w'(\theta,z,e)}{w(\theta_0,z,e)}.
\end{align}
Then, Proposition \ref{prop2} and the non-negativity of $r$ imply that
\begin{align}
\undertilde{\alpha}(\theta_0) \overset{\Delta}{=} \frac{\undertilde{w}(\theta_0,z,r)}{\undertilde{w}(\theta_0,z,e)+\mu_1(e)} \leq \alpha(\theta_0) \leq \frac{\undertilde{w}(\theta_0,z,r)+\mu_1(r)}{\undertilde{w}(\theta_0,z,e)} \overset{\Delta}{=} \tilde{\alpha}(\theta_0) 
\label{eq:eqn5_2}
\end{align}
and
\begin{align}
 \undertilde{w}(\theta_0,z,e) \leq w(\theta_0,z,e) \leq \undertilde{w}(\theta_0,z,e)+\mu_1(e).
\end{align}
On the other hand, Theorem \ref{thm2} and (\ref{eq:eqn5_2}) yield the bounds
\begin{align}
  l' \leq \alpha(\theta_0) w'(\theta_0,z,e) \leq u',
\end{align}
where
\begin{align*}
 l' = \begin{cases}
      \undertilde{\alpha}(\theta_0) (\undertilde{w}'(\theta_0,z,e)-\mu_1'(e)), & \undertilde{w}'(\theta_0,z,e)-\mu'_1(e) \geq 0 \\ 
      \tilde{\alpha}(\theta_0) (\undertilde{w}'(\theta_0,z,e)-\mu_1'(e)), & \undertilde{w}'(\theta_0,z,e)-\mu'_1(e) < 0  
      \end{cases}
\end{align*}
and
\begin{align*}
 u' = \begin{cases}
      \tilde{\alpha}(\theta_0) (\undertilde{w}'(\theta_0,z,e)+\mu_1'(e)), & \undertilde{w}'(\theta_0,z,e)+\mu'_1(e) \geq 0 \\ 
      \undertilde{\alpha}(\theta_0) (\undertilde{w}'(\theta_0,z,e)+\mu_1'(e)), & \undertilde{w}'(\theta_0,z,e)+\mu'_1(e) < 0.  
      \end{cases}
\end{align*}
Another application of Theorem \ref{thm2} then yields the bounds
\begin{align}
 \undertilde{\alpha}'(\theta_0) \leq \alpha'(\theta_0) \leq \tilde{\alpha}'(\theta_0),
\label{eq:eqn5_5}
\end{align}
where
\begin{align*}
\undertilde{\alpha}'(\theta_0) = \begin{cases}
                                \frac{\undertilde{w}'(\theta_0,z,r)-\mu_1'(r)-u'}{\undertilde{w}(\theta_0,z,e)+\mu_1(e)},& \quad \undertilde{w}'(\theta_0,z,r)-\mu_1'(r)-u'\geq 0,\\
                                \frac{\undertilde{w}'(\theta_0,z,r)-\mu_1'(r)-u'}{\undertilde{w}(\theta_0,z,e)},& \quad \undertilde{w}'(\theta_0,z,r)-\mu_1'(r)-u'<0
                                \end{cases}
\end{align*}
and
\begin{align*}
\tilde{\alpha}'(\theta_0) = \begin{cases}
                                \frac{\undertilde{w}'(\theta_0,z,r)+\mu_1'(r)-l'}{\undertilde{w}(\theta_0,z,e)},& \quad \undertilde{w}'(\theta_0,z,r)+\mu_1'(r)-l'\geq 0,\\
                                \frac{\undertilde{w}'(\theta_0,z,r)+\mu_1'(r)-l'}{\undertilde{w}(\theta_0,z,e)+\mu_1(e)},& \quad \undertilde{w}'(\theta_0,z,r)+\mu_1'(r)-l'<0.
                                \end{cases}
\end{align*}
The inequalities (\ref{eq:eqn5_5}) provide our bounds.

\section{Extension to Markov Jump Processes}\label{sec:extension_markov_jump_process}

We now briefly discuss the extension of the results of Sections \ref{sec:trunc_approx_grads} through \ref{sec:errorbound_deriv} to the setting of Markov jump processes. In particular, suppose that $(X(t): t\geq 0)$ is an irreducible  Markov jump process with a rate matrix $Q(\theta) = (Q(\theta,x,y): x,y \in S)$ depending on a scalar parameter $\theta$. 
Let $R(\theta) = (R(\theta,x,y): x,y \in S)$ be the stochastic matrix for which
\begin{align*}
 R(\theta,x,y) = \begin{cases}
                 Q(\theta,x,y)/\lambda(\theta,x)&, x \neq y\\
                 0&, x = y
                 \end{cases}
\end{align*}
where $\lambda(\theta,x) = -Q(\theta,x,x)$. Suppose that $R(\theta)$ is recurrent, thereby implying that $(X(t): t\geq 0)$ is non-explosive and recurrent under $Q(\theta)$; see for example \cite{breiman1968probability}.

If we assume that $(X(t): t\geq 0)$ is positive recurrent under $Q(\theta)$, then $Q(\theta)$ has a unique stationary distribution $\pi(\theta) = (\pi(\theta,y): y \in S)$ satisfying $\pi(\theta) Q(\theta) = 0$, where $\pi(\theta)$ is encoded as a row vector. For a given reward function $\tilde{r}: S \to \mathbb{R}_+$ (encoded as a column vector), the equilibrium reward per unit times is given by $\alpha(\theta) = \pi(\theta) \tilde{r}$. Furthermore, if $z \in S$ is fixed, 
\begin{align*}
  \alpha(\theta) = \frac{\mathbb{E}_z^{\theta} \sum_{j=0}^{\tau(z)-1} r_1(\theta,Y_j)}{\mathbb{E}_z^{\theta} \sum_{j=0}^{\tau(z)-1} r_2(\theta,Y_j)}
\end{align*}
where $(Y_j: j\geq 0)$ is the discrete time Markov chain having transition matrix $R(\theta)$, $r_1(\theta,y)  = \tilde{r}(y)/\lambda(\theta,y)$ and $r_2(\theta,y) = 1/\lambda(\theta,y)$.
Hence,
\begin{align*}
  \alpha(\theta) = \frac{w(\theta,z,r_1(\theta))}{w(\theta,z,r_2(\theta))},
\end{align*}
so that
\begin{align*}
  \alpha'(\theta_0) = \frac{w'(\theta_0,z,r_1(\theta_0))+w(\theta_0,z,r'_1(\theta_0)) - \alpha(\theta_0)(w'(\theta_0,z,r_2(\theta_0))+w(\theta_0,z,r'_2(\theta_0)))}{w(\theta_0,z,r_2(\theta_0))}.
\end{align*}
Note that Proposition \ref{prop2} bounds the error in $w(\theta_0,z,r'_i(\theta_0))$ and Theorem \ref{thm2} bounds the error in $w'(\theta_0,z,r_i(\theta_0))$ for $i=1,2$. We can then similarly bound the error in $\alpha'(\theta)$ via an argument similar to that used in Section \ref{sec:errorbound_deriv}.

\section{Numerical Results}

In this section, we apply our gradient bounding method to two examples and provide numerical results.

\subsection{Single Queue}

We consider the $G/M/1$ queue (see p.279 of \citet{asmussen2003applied}), specifically the Markov chain $(X_n:n \geq 0)$ in which $X_n$ corresponds to the number-in-system just prior to the $(n+1)'st$ arrival epoch. We let $F$ be the cumulative distribution function for the inter-arrival times and let $\mu$ be the rate parameter associated with the exponentially distributed service times. Here, the state space $S = \{0,1,2,\ldots\}$. We have 
\begin{align*}
 X_{n+1} = \left[X_n + 1 - Z_{n+1}  \right]^+
\end{align*} 
where the $Z_i$'s are independent and identically distributed, and 
\begin{align*}
 \xi_k = \mathbb{P}(Z_1 = k) = \int_{[0, \infty)} e^{-\mu t} \frac{(\mu t)^k}{k!} F(dt)
\end{align*}
for $k \geq 0$.

We set $F$ to be the uniform distribution over the interval $[0,b]$ and $r(x) = x$ for $x \in S$. Assuming $\mu b >2$, the equilibrium distribution of $X$ is $\pi(n) = (1 - \beta) \beta ^n$ for $n \geq 0$, where $\beta = \beta (\mu) = 1 - \phi(\mu)/{\mu}$ and $\phi = \phi(\mu) \in [0,\mu]$ is the solution to the following equation:
\begin{align*}
  1 - \frac{\phi}{\mu} = \int_0^{\infty} e^{-\phi t} dF(t) = \frac{1 - e^{-\phi b}}{\phi b}.
\end{align*}
Moreover,  
\begin{align*}
 \xi_k = \frac{1}{ \mu  b} \left\{1 - \sum_{i=0}^k \frac{(\mu b)^i}{i!} e^{-\mu b} \right\}.
\end{align*}

Viewing the Markov chain as a function of $\mu$, we focus on taking derivatives with respect to $\mu$. We have
\begin{align*}
 \alpha(\mu) = \frac{\beta(\mu)}{1 - \beta(\mu)}, \alpha'(\mu) = \frac{\beta'(\mu)}{(1 - \beta(\mu))^2}
\end{align*}
where
\begin{align*}
  \beta'(\mu) &= \frac{\phi(\mu) - \mu \phi'(\mu)}{\mu^2},\\
  \phi'(\mu)  &= \left(b \cdot e^{-\phi b} + \frac{2b \phi}{\mu} - b \right)^{-1} b \frac{\phi^2}{\mu ^2}. 
\end{align*}
If we set $b = 1.5$ and $\mu_0 = 2$, we find that
\begin{align*}
  \alpha(\mu_0) = 1.2046186974, \alpha'(\mu_0) = -1.9530450987.
\end{align*}

To test our methodology, we set $K = \{0,1,2,\ldots,9 \}$ and take the truncation set $A_n = \{0,1,2,\ldots,  n\}$ for $n \in \mathbb{N}$.
If we choose $v(x) = 2 x^2$, then
\begin{align*}
&v(x) = 2 \mathbb{E}\left([x+1-Z_1]^+\right)^2 \leq 2 \mathbb{E}\left(x+1-Z_1\right)^2 = \\
& v(x) - r(x) + (4(1 - \mathbb{E}Z_1)+1) x + 2 (1 + \mathbb{E}Z^2_1 - 2 \mathbb{E}Z_1 ).  
\end{align*}
So $v(x) \leq v(x) - r(x) + c I(x \in K)$ for $c = 2 (1 + \mathbb{E}Z^2_1 - 2 \mathbb{E}Z_1 )$  as $4(1 - \mathbb{E}Z_1)+1 <0$.
In addition, we set $\tilde{v}(x) = x^4$.
Based on Remark \ref{rmk_1}, we set $w = e_2$ and $\nu = e_{\beta}$.
Further, we set $\tilde{\nu}(x)= x^4$ and $z = 0$.

Figure \ref{fig:fig1} shows the numerical accuracy of our method versus $A_n$ as $n$ varies. For this figure,
\begin{align*}
  lower\_rel = \frac{\alpha' - lb}{|\alpha'|}, upper\_rel = \frac{ub - \alpha'}{|\alpha'|}, gap\_rel = \frac{ub - lb}{|\alpha'|}
\end{align*}
where $ub, lb$ are the upper and lower bounds obtained from out gradient bounds method.
The figure indicates our method can provide highly accurate bounds for the gradient $\alpha'(\mu_0)$ with moderate size truncation sets.

\begin{figure}[htbp]
\centering
\includegraphics[scale=0.4]{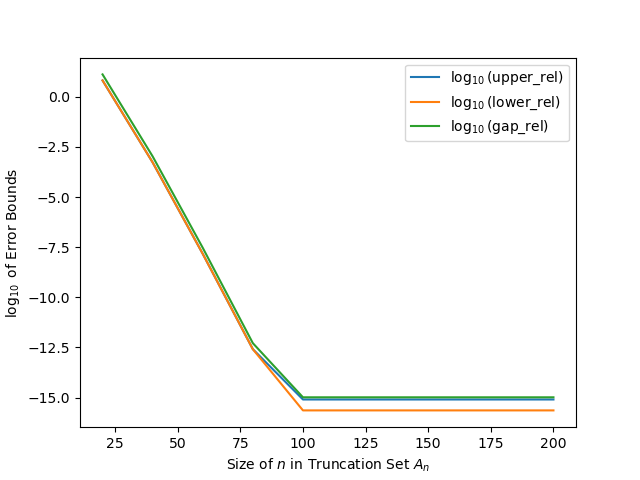}
\caption{ G/M/1 queue: Numerical accuracy of our gradient bounds method versus $n$}
\label{fig:fig1}
\end{figure}

\subsection{A Network of Queues}

Consider a Jackson network consisting of two stations with a single server at each station. The queues are indexed with $1,2$. The routing matrix of the network is 
\begin{align}
 R(\theta) =  \begin{bmatrix}
   1/5 & \theta \\
   1/4 & 1/8 
  \end{bmatrix}.
\end{align}
The service rate of the server at station $1$ is $\mu_1 = 3.0$, while the service rate for the server at station 2 is $\mu_2 = 3.0$. The exogenous arrival rates for the stations are $\lambda_1= 2/3, \lambda_2= 1$. Define $X(t) = (X_1(t),X_2(t))$ where $X_i(t)$ is the number of customers at station $i$ at time $t$. The process $X = (X(t): t \geq 0)$ is a Markov jump process with state space $S = \{x = (x_1,x_2): x_1, x_2\geq 0 \}$.
We set $\tilde{r}(x) = x_1 + x_2$. The function $\tilde{r}$ was defined in Section \ref{sec:extension_markov_jump_process}. 

We focus on taking derivatives with respect to $\theta$ and set $\theta_0 = 0.4$. We have
\begin{align*}
 \alpha(\theta) &= \frac{\rho_1(\theta)}{1 - \rho_1(\theta)} +\frac{\rho_2(\theta)}{1 - \rho_2(\theta)}, \\ 
 \alpha'(\theta) &= \frac{\rho'_1(\theta)}{(1 - \rho_1(\theta))^2} +\frac{\rho'_2(\theta)}{(1 - \rho_2(\theta))^2} 
\end{align*} 
where $\rho_i(\theta) = \gamma_i(\theta)/{\mu_i}$ for $i \in {1,2}$ and $\gamma_i(\theta)$ is the effective arrival rate for queue $i$. The effective arrival rates are the solution of the linear system of equations 
\begin{align*}
\begin{bmatrix} \gamma_1(\theta) \\ \gamma_2(\theta) \end{bmatrix} = \begin{bmatrix} 2/3 \\ 1 \end{bmatrix}+\begin{bmatrix} 1/5 & 1/4 \\\theta & 1/8 \end{bmatrix} \begin{bmatrix} \gamma_1(\theta) \\ \gamma_2(\theta) \end{bmatrix}.
\end{align*}
We get 
\begin{align*}
 \alpha(\theta_0) = 2.316614420063, \alpha'(\theta_0) = 4.38785487564.
\end{align*} 

We set $K = \{(x_1,x_2): (0.9667)x_1 + (0.6999)x_2 \leq 13.4666 \}$. 
Also we set the truncation set $A_n = \{(x_1,x_2)\in S: x_1 \leq n , x_2 \leq n \}$.
Moreover, for $r_1$ and $r_2$ (defined in Section \ref{sec:extension_markov_jump_process}), we choose the same Lyapunov functions  $v(x) = x_1^2 + 2 x^2_2$ and $\tilde{v}(x) = 60 x^3_1 + 60 x^3_2$.
Our choice of the set $K$ is the result of writing the Lyapunov inequality for function $v$. In particular for $x_1,x_2 \geq 1$,
\begin{align*}
(Q v) (x) =& -\tilde{r}(x) + (2(\lambda_1 - \mu_1 (1 - R_{11}) + \mu_2 R_{21} )+1) x_1 \\ &+ (4(\lambda_2 - \mu_2(1-R_{22}) + \mu_1 R_{12} ) + 1) x_2 \\
&+ (\lambda_1 + \mu_1(1-R_{11}) + \mu_2 R_{21}) + 2(\lambda_2 + \mu_2(1-R_{22}) + \mu_1 R_{12}).
\end{align*}
This set $K$ makes the required Lyapunov conditions mentioned above for both $v$ and $\tilde{v}$ valid.
Again based on Remark \ref{rmk_1}, we set $w = e_2$ and $\nu = e_{\beta}$. Further, we set $\tilde{\nu}(x) = 60 x^3_1 + 60 x^3_2$ and $z=(0,0)$.

Figure \ref{fig2} provides the numerical performance of our gradient bounding method as a function of $n$. Again, the figure indicates that our method can provide highly accurate bounds for the gradient $\alpha'(\theta_0)$ with moderate size truncation sets.

\begin{figure}[htbp]
\centering
\includegraphics[scale=0.4]{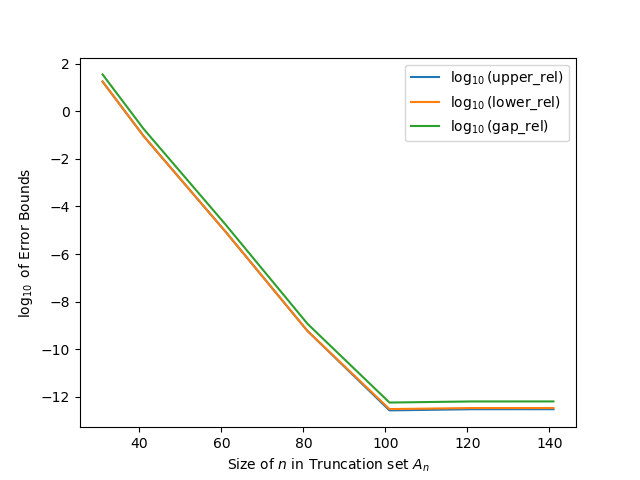}
\caption{Network of queues: Numerical accuracy of our gradient bounds  versus $n$.}
\label{fig2}
\end{figure}

\bibliographystyle{apalike}
\bibliography{papers}

\end{document}